\documentclass[11pt,a4paper,oneside]{amsart}
\usepackage{amssymb,amsmath,mathrsfs,amsthm,amsrefs,graphicx,color,bm, mathtools}
\usepackage{leftidx}

\theoremstyle{plain}
\newtheorem{theorem}{Theorem}[section]
\newtheorem{definition}[theorem]{Definition}
\newtheorem{lemma}[theorem]{Lemma}
\newtheorem{corollary}[theorem]{Corollary}
\newtheorem{proposition}[theorem]{Proposition}

\theoremstyle{remark}
\newtheorem{remark}[theorem]{Remark}

\allowdisplaybreaks[4]

\usepackage{hyperref}

\numberwithin{equation}{section}

\makeatletter
\@namedef{subjclassname@2020}{\textup{2020} Mathematics Subject Classification}
\makeatother

\newcommand{\C}{\mathbb{C}}
\newcommand{\R}{\mathbb{R}}
\newcommand{\Z}{\mathbb{Z}}
\newcommand{\N}{\mathbb{N}}

\newcommand{\F}{\mathcal{F}}

\renewcommand{\Im}{\operatorname{Im}}
\renewcommand{\Re}{\operatorname{Re}}

\newcommand{\norm}[1]{\left\lVert #1\right\rVert}

\def\({\left(}
\def\){\right)}
\def\<{\left\langle}
\def\>{\right\rangle}
\def\le{\leqslant}
\def\ge{\geqslant}

\def \F{\mathcal{F}}

\DeclareMathOperator{\im}{Im}
\DeclareMathOperator{\re}{Re}

\newcommand{\todayd}{\the\year/\the\month/\the\day}

\newcommand{\Hsp}{(H^1({\mathbb{R}}^3))^N}
\newcommand{\IM}{I_{\text{max}}}
\newcommand{\TM}{T_{\text{max}}}
\newcommand{\Tm}{T_{\text{min}}}
\newcommand{\uv}{\bm{u}}
\newcommand{\wv}{\bm{w}}
\newcommand{\uo}{\bm{u}_0}

\newcommand{\uc}{\bm{u}_c}
\newcommand{\wzero}{\bm{w}_n(0)}
\newcommand{\Pv}{\bm{\Phi}}
\renewcommand{\utilde}{\tilde{\bm{u}}}
\newcommand{\Fl}{U(t-t_0)\bm{u}_0}
\newcommand{\FnlF}{\displaystyle \int_{t_0}^{t} U(t-s) \bm{F}(\bm{u}(s))ds}
\newcommand{\ep}{\varepsilon}
\newcommand{\phiv}{\bm{\phi}}
\newcommand{\Qv}{\bm{Q}}
\newcommand{\lims}{\varlimsup}
\newcommand{\Xnorm}[2]{\| {#1} \| _{#2}}
\newcommand{\inner}[2]{\langle {#1}, {#2}\rangle}
\newcommand{\ctext}[1]{\raise0.2ex\hbox{\textcircled{\scriptsize{#1}}}}
\newcommand{\br}[1]{\{ {#1} \}}
\newcommand{\Lsp}[3]{L^{#1}({#2};{#3})}
\newcommand{\Lpq}[2]{L^{#1}L^{#2}}%
\newcommand{\Equ}[1]{i\partial_t{#1} + \Delta {#1}}
\newcommand{\rv}[1]{\bm{r}_n^{#1}}
\newcommand{\tn}[1]{t_n^{#1}}
\newcommand{\yn}[1]{y_n^{#1}}
\newcommand{\UT}[1]{(U(t_n^{#1})T_{y_n^{#1}})}
\newcommand{\UTinv}[1]{(U(t_n^{#1})T_{y_n^{#1}})^{-1}}
\newcommand{\PN}[1]{P_{{#1}N}}
\newcommand{\Finverse}[1]{\mathcal{F}^{-1} \left[{#1} \right]}

\newcommand{\Ty}[1]{T_{#1}^{-1}}

\begin{document}
\title[]
{Scattering below ground states for a class of systems of nonlinear Schrodinger equations}

\author[S. Masaki]{Satoshi MASAKI}
\address[]{Division of Mathematical Science, Department of Systems Innovation, Graduate School of Engineering Science, Osaka University, Toyonaka, Osaka, 560-8531, Japan}
\email{masaki@sigmath.es.osaka-u.ac.jp}
\author[R. Tsukuda]{Ryusei TSUKUDA}
\address[]{Division of Mathematical Science, Department of Systems Innovation, Graduate School of Engineering Science, Osaka University, Toyonaka, Osaka, 560-8531, Japan}
\email{r-tsukuda@sigmath.es.osaka-u.ac.jp}
\keywords{}
\subjclass[2020]{Primary 35Q55; Secondary 35B40 }


\begin{abstract}
In this paper, we consider the scattering problem for a class of $N$-coupled systems of the cubic nonlinear 
Schr\"odinger equations in three space dimensions.
We prove the scattering of solutions which have a mass-energy quantity 
less than that for the ground states.
This result is previously obtained by Duyckaerts-Holmer-Roudenko for the single cubic nonlinear 
Schr\"odinger equation in three space dimensions.
It turns out that the result can be extended to a wide class of $N$-coupled systems.
\end{abstract}

\maketitle

\section{Introduction}
In this paper, we study the Cauchy problem of the following $N$-coupled system of nonlinear
Sch\"{o}rdinger equations:
\begin{equation}\label{E:gNLS}
\left\{
\begin{aligned}
& \, \Equ{u_j} + F_j(\uv) =0, \quad 1 \le j \le N, \quad  (t,x)\in \R \times \R^3,  \\
& \, \uv(t_0)=(u_{1,0},u_{2,0},\dots,u_{N,0})\in (H^1(\R^3))^N, \\
\end{aligned}
\right.
\tag{gNLS}
\end{equation}
where $N \ge 2$ and $\uv=\uv(t,x)=(u_1(t,x),u_2(t,x),\dots,u_N(t,x))$ is a $\C^N$-valued unknown.
$\Delta = \frac{\partial^2}{\partial x_1^2} + \frac{\partial^2}{\partial x_2^2}+ \frac{\partial^2}{\partial x_3^2}$
is the standard Laplacian.
$F_j$ is a cubic nonlinearity. 
We assume that the nonlinearity  is given by the relation
\begin{align}\tag{HS}\label{E:Hamiltonian}
  F_j := \tfrac{1}{2}\partial_{\overline{z_j}} g 
\end{align}
from a $\R$-valued quartic polynomial $g \colon \C^N \to \R$ with real coefficients which satisfies
the gauge condition
\begin{align}\label{gauge condition}\tag{GC}
  g(e^{i\theta}z_1,e^{i\theta}z_2,\dots,e^{i\theta}z_N)=g(z_1,z_2,\dots,z_N) 
\end{align}
for all $(z_1,z_2,\dots,z_N) \in \C^N$ and $\theta \in \R$.
One consequence of assumption \eqref{E:Hamiltonian} is that \eqref{E:gNLS} possesses the Hamiltonian structure. 
Indeed, \eqref{E:gNLS} has the conserved energy
\begin{equation}\label{E:energy}
	  E(\uv) := \int_{\R^3} 
  \left(
    \frac{1}{2}\sum_{j=1}^N |\nabla u_j|^2 - \frac{1}{4} g(\uv)
    \right)dx,
\end{equation}
at least formally. By \eqref{gauge condition} and \eqref{E:Hamiltonian},
there are two more conserved quantities:
  \begin{align*} 
  & M(\uv) := \int_{\R^3} 
   \frac{1}{2}\sum_{j=1}^N | u_j|^2 
 dx,&
  &  P(\uv) := \sum_{j=1}^N  \im{
    \int_{\R^3} 
   \overline{u}_j \nabla u_j 
 dx
  }.
  \end{align*}
They correspond to the mass (or charge) and the moment, respectively.
We denote $\bm{F}(\uv)=(F_1(\uv),F_2(\uv),\dots,F_N(\uv))$.
For the existence of bound states, we suppose
\begin{align}\label{gstate}
  g_{\text{max}} := \max_{  |z_1|^2 + |z_2|^2 + \dots |z_N|^2=1} g(z_1,z_2,\dots,z_N) >0.
\end{align}

The system \eqref{E:gNLS} contains several well-known systems.
The first one is the Manakov system or the vector valued  nonlinear
Sch\"{o}rdinger equation:
\begin{equation*}
   \Equ{u_j} +    u_j \sum_{j=1}^N |u_j|^2 =0, \quad 1 \le j \le N, \quad  (t,x)\in \R \times \R^3.
\end{equation*}
This system is given by $g(u_1,u_2,\dots,u_N)=(\sum_{j=1}^N |u_j|^2)^2$.
It is known that the system is an integrable system in one space dimension \cite{Manakov} (see also \cites{APT1,APT2}).
The second example is as follows:
\begin{equation*}
  \left\{
    \begin{aligned}
      & \, \Equ{u_1} +(a+b)(|u_1|^2 + |u_2|^2)u_1 + (a-b)|u_3|^2u_1 + bu_2^2\overline{u_3} =0, \\
      & \, \Equ{u_2} + (a+b)(|u_1|^2 + |u_3|^2)u_2 + a|u_2|^2u_2 +2 bu_1\overline{u_2}u_3 =0, \\ 
      & \, \Equ{u_3} + (a+b)(|u_2|^2 + |u_3|^2)u_3 + (a-b)|u_1|^2u_3 + b\overline{u_1}u_2^2 =0, \\ 
    \end{aligned}
  \right. 
\end{equation*}
where $(t,x)\in \R \times \R^3$ and  $a,b \in \R$ are constants.
This is a model system in the spinor Bose-Einstein condensate  (see \cite{Wadati-Tsuchida}, for instance). 
This system corresponds to the choice $N=3$ and
\begin{align*}
	g(z_1,z_2,z_3) ={}& a(|z_1|^2 + |z_2|^2 + |z_3|^2)^2 \\&+ b((|z_1|^2-|z_3|^2)^2  + 2 |z_2|^2 (|z_1|^2+|z_3|^2) + 4 \Re (\overline{z_1} z_2^2 \overline{z_3})).
\end{align*}
One easily deduces from this form that there exists $a_0=a_0(b)$ such that \eqref{gstate} holds for $a>a_0$.

The existence and the characterization of the ground states of \eqref{E:gNLS} are studied by the first author in \cite{Masaki}. 
The instability of the ground states and
a sharp criteria for the global existence are also obtained by using the variational  characterization of the ground states.
The global existence result is a global-existence-below-ground-state type result.
The aim of this paper is to establish the scattering of the global solutions.

Before going into details, we briefly recall the results for the single NLS equation:
\begin{align}\label{E:NLS}
  \left\{
  \begin{aligned}
    & \, \Equ{u} + \mu |u|^{p-1}u =0,\ \ t \in \R,\, x \in \R^d,  \\
    & \, u(0,x)=u_0 \in H^1(\R^d),
    \end{aligned}
  \right.
  \tag{NLS}
\end{align}
where $d\ge 1$, and $\mu = \pm 1$, and $1+\frac4d \le p<2^*-1$ with $2^*=\infty$ if $d=1,2$ and $2^*=\frac{2d}{d-2}$ if $d\ge 3$.
The local well-posedness in $H^1(\R^d)$ is well-known (see \cite{Cazenave}, for instance).
For an $H^1$-solution $u(t)$, the following two quantities are conserved:
\begin{align*} 
  M(u) := \int_{\R^d} 
    \frac{1}{2}|u|^2
  dx,\ \ \ 
  E(u) := \int_{\R^d} 
  \left(
    \frac{1}{2}|\nabla u|^2 - \frac{1}{p+1} |u|^{p+1}
    \right)dx
\end{align*}
We say $u(t)$ scatters for positive time direction (resp. negative time direction) if
$u(t)$ exists globally forward (resp. backward) in time and 
there exists $u_{+} \in H^1$ (resp. $u_{-} \in H^1$) such that
\begin{align*}
   & \lim_{t \to  \infty} \norm{U(-t)u(t) - u_{+}}_{H^1}=0 &
    \left( \text{resp. }  \lim_{t \to - \infty} \norm{U(-t)u(t) - u_{-}}_{H^1}=0 \right)
\end{align*}
holds. This definition is naturally extended to the system case.

There are large number of literature on the scattering problem of \eqref{E:NLS}.
When $\mu = -1$, the nonlinearity is repulsive and all solutions scatter for both time directions
(see \cite{Ginibre-Velo, Bourgain99,Collianderetal, Ryckman-Visan, Tao, Visan, Killip-Tao-Visan, Tao-Visan-Zhang, Dod3d,Dod2d,Dod1d}). 
On the other hand, there exists a non-scattering solution when $\mu = 1$. 
A typical example of the non-scattering solution is a soliton.
Let $Q=Q(x)$ be the positive radial solution to $- \Delta Q + Q = |Q|^{p-1} Q$ on $\R^d$. 
It is known that one can obtain a sharp sufficient condition for scattering by means of variational characterization of $Q$.
As for the mass-supercritical and energy-subcritical case $1+ \frac{4}{d} < p < 1 + \frac{4}{d-2}$,
this kind of result is established in
Duyckaerts-Holmer-Roudenko \cite{Duyckaerts-Holmer-Roudenko} and Akahori-Nawa \cite{Akahori-Nawa}
(see also \cite{Holmer-Roudenko, Fang-Xie-Cazenave}). 
In particular, in the 3D cubic case ($p=3$ and $d=3$), the following is obtained
\begin{theorem}[\cite{Duyckaerts-Holmer-Roudenko}]\label{T:DHR}
Let $p=3$, $d=3$, and $\mu=1$.
If the initial data $u_0 \in H^1(\R^3)$ satisfies
  \begin{align*}
    &\norm{u_0}_{L^2} \norm{\nabla u_0}_{L^2} < \norm{Q}_{L^2} \norm{\nabla Q}_{L^2}, \\
    &M(u)E(u)<M(Q)E(Q)  
  \end{align*}
  then the solution $u$ to \eqref{E:NLS} scatters for both time directions.
\end{theorem}
We remark that a sharp blowup result is also obtained in these studies.
Let us concentrate on the scattering problem in this paper.

Theorem \ref{T:DHR} is obtained by concentration-compactness/rigidity-type argument
 initiated by  Kenig and Merle \cite{Kenig-Merle}
(see \cite{Dodson-Benjamin:radial, Dodson-Murphy} for a proof without concentration compactness). 
In \cite{Kenig-Merle,Killip-Visan, Dodson-Benjamin}, the energy-critical case $d\ge 3$ and $p=2^*-1$ is studied. 
This approach is also used in the mass-critical case $d\ge 1$ and
$p= 1 + \frac{4}{d}$ (\cite{Killip-Tao-Visan, Killip-Visan-Zhang,Dodson}). 
In the above theorem, the mass-energy quantity of solution is assumed to be smaller than that of the ground state.
The case where it is slightly above that of the ground state is studied in \cite{Nakanishi-Schlag} (see also \cite{Duyckaerts-Roudenko} for the threshold case).

This kind of problems are studied also for NLS systems. See, for instance, \cite{Cheng-Guo-Hwang-Yoon,Inui-Kishimoto-Nishimura, Hamano,Hamano-Inui-Kishimoto}. 
By these studies, it is revealed that the so-called mass-resonant condition plays an important role.
This is a condition about the coefficients of the Laplacian in the linear part of equations in a system.
The validity of the condition corresponds to that of the Galilei invariance.
We just point out that \eqref{E:gNLS} satisfies the condition and
omit a precise description of the condition.

\subsection{Main result}
Let us make notations.
Let $L^p(\R^d)$ ($1\le p \le \infty$) and $H^1(\R^d)$ be the standard Lebesgue and Sobolev spaces on $\R^d$, respectively.
For a Banach space $X$ of $\C$-valued functions, $X^N$ stands for the set of $\C^N$-valued functions 
to which each component belongs to $X$. We define the norm of $X^N$ as follows: 
For ${\bf f}=(f_1,\dots,f_N) \in X^N$ with $f_j \in X$ for $1 \le j \le N$,
$\norm{\bf f}_{X^N} := (\sum_{j=1}^N \norm{f_j}_X^2)^{1/2}$.

\begin{definition}[an $H^1$-solution to \eqref{E:gNLS}]
  Let $I \subset \R$ be an interval and let $t_0 \in \bar{I}$. 
  We say a function $\uv(t)$ is an $H^1$-solution to \eqref{E:gNLS} on $I$ if $\uv(t)$ belongs to
$    {(C(I;H^1(\R^3)) \cap L^{2}(I;W^{1,6}(\R^3)))}^2
 $
  and satisfies
  \begin{align*} 
    u_j(t)=U(t-t_0)u_{j,0} + i \displaystyle \int_{t_0}^{t} U(t-s) F_j(\bm{u}(s))ds
  \end{align*}
  on $I$ for all $j=1,2,\dots,N$. 
\end{definition}

Note that if $\bm{\Phi} = (\Phi_1,\Phi_2,\dots,\Phi_N) \in (H^1(\R^d))^N$ is a solution to the nonlinear elliptic system
\begin{align}\label{E:gE}
  - \Delta \Phi_j + \omega \Phi_j = F_j(\bm{\Phi}),  \quad 1 \le j \le N, \quad  x\in \R^3,  
    \tag{gE$_\omega$}
\end{align}
for some $\omega>0$
then the function $(e^{i \omega t}\Phi_1,e^{i \omega t}\Phi_2, \dots, e^{i \omega t}\Phi_N)$ is an exact $H^1$-solution to \eqref{E:gNLS}.
They are one-soliton solutions to \eqref{E:gNLS}.

Let us now recall the characterization of ground state given in \cite{Masaki}.
As mentioned above, under the assumption \eqref{gstate}, \eqref{E:gNLS} admits ground state solutions.
Roughly speaking, a ground state is a solution to \eqref{E:gE} which has the minimum action
(see Definition $\ref{def.5.2}$ for the precise definition). 
Let  $\mathcal{G}$ be the set of ground state. 
By \cite{Masaki}*{Theorem 1.2}, the set $\mathcal{G}$ is given as follow.
Let $Q_{\omega, g_{\max}}:= (\omega/g_{\max})^{\frac1{2}}Q(\sqrt\omega \cdot) \in H^1(\R^3)$,
where $Q \in H^1(\R^3)$ is the positive radial solution to
$  - \Delta Q + Q = Q^3 $ on $\R^3$.
Note that $Q_{\omega, g_{\max}}(x) $ solves
$  - \Delta Q + \omega Q = g_{\text{max}} Q^3 
$
on $\R^3$. Then, the set of ground state is given as $\mathcal{G}=\cup_{\omega>0} \mathcal{G}_\omega$ with
\begin{align}\label{E:Gomega}
  \mathcal{G}_\omega = \bigcup_{{\bf w} \in T_0} \br{{\bf w} Q_{\omega,g_{\min}}(\cdot - y) \in  (H^1(\R^3))^N \mid y \in \R^3},
\end{align}
where
\begin{align*}
  T_0 := \left\{{\bf z} =(z_1,z_2,\dots,z_N) \in \C^N \ \middle|\ \sum_{j=1}^N |z_j|^2=1,\ g({\bf z})=g_{\text{max}}\right\}.
\end{align*}

Our main result is the following.
\begin{theorem}\label{main}
 If the data $\uo \in \Hsp$ satisfies
  \begin{align}
    & M(\uo)E(\uo) < M(\Qv)E(\Qv),\label{condition1}\\
    & \norm{\uo}_{(L^2)^N} \norm{\nabla \uo}_{(L^2)^N}< \norm{\Qv}_{(L^2)^N} \norm{\nabla \Qv}_{(L^2)^N}\label{condition2}
  \end{align}
  then the corresponding $H^1$-solution $\uv$ to \eqref{E:gNLS} scatters for both time directions,
  where $\Qv \in \mathcal{G}$.
\end{theorem} 
\begin{remark}
By \eqref{E:Gomega}, one has $M(\Qv)E(\Qv)=g_{\max}^{-2}M(Q)E(Q)$ and $\norm{\Qv}_{(L^2)^N} \norm{\nabla \Qv}_{(L^2)^N}= g_{max}^{-1} \norm{Q}_{L^2} \norm{\nabla Q}_{L^2}$ for all $\Qv \in \mathcal{G}$, where $Q$ is the (normalized) ground state in the single case.
\end{remark}
The proof of the main theorem is done by extending the argument in \cite{Duyckaerts-Holmer-Roudenko} to the system setup.
We use the concentration-compactness/rigidity argument by Kenig and Merle.
It will turn out that the proof for the single case works with a slight modification.
It can be said that
a large class of NLS systems is handled in a unified way by our formulation.

The outline of the proof is as follows.
We first reformulate the above theorem as a variational problem.
The starting point is that, under the assumption \eqref{condition2}, 
a solution scatters if the value of the mass-energy quantity of the solution is sufficiently small.
Then, let $A_c$ be the critical value of the mass-energy quantity such that 
the scattering of solutions is assured together with \eqref{condition2}.
One sees by definition that this value does not exceed the value of the mass-energy quantity of the ground state,
that is, one has $A_c \le M(\Qv)E(\Qv)$. 
Then, the matter is reduced to showing that
\begin{align*}
  A_c = M(\Qv)E(\Qv).
\end{align*} 
We prove this by contradiction. 
If it fails then one can show the existence of a critical element $\uc(t)$.
This $\uc(t)$ is a global non-scattering solution to \eqref{E:gNLS}
 of which value of the mass-energy quantity is exactly $A_c$ and 
 of which orbit is precompact modulo space translation.
We use a profile decomposition to prove this.
Finally, we control the space translation by the analysis of the moment of the critical element
and apply the truncated virial estimate to deduce a contradiction.

The rest of the paper is organized as follows.
In Section \ref{S:lwp}, we collect basic facts on \eqref{E:gNLS}.
In particular, we obtain the local well-posedness, the equivalent characterization of the scattering, and the long time perturbation.
In Section \ref{S:pd}, we prove the profile decomposition.
The property of the ground states are briefly recalled in Section \ref{S:gs}.
Then, we prove the main result in Section \ref{S:proof}.

\section{Local well-posedness and scattering}\label{S:lwp}

In this section we briefly recall basic results. We omit the proof since they follow by a standard argument
(see \cite{Cazenave}. for instance).

\subsection{Function space and nonlinear estimate}

  For an interval $I \subset \R$, we define $S(I):=\Lsp{8}{I}{(L^4(\R^3))^N}$ and $N(I) := L^\frac{8}{3}(I;(L^\frac{4}{3}(\R^3))^N)$. 
  Note that the $S$-norm is invariant under the scaling $u(t,x) \mapsto \lambda u(\lambda^2 t, \lambda x)$.  
  By Strichartz's estimates for non-admissible pair (See \cite{Kato, Vilela, Koh}), we have
  \begin{eqnarray} \label{SN esti.}
    \Xnorm{\FnlF}{S(I)} \le C   \Xnorm{\bm{F}(\uv)}{N(I)}.
  \end{eqnarray}
  Also by the definition of $F_j$ and using H\"older's inequality, we also have
  \begin{eqnarray} \label{NS esti.}
    \Xnorm{\bm{F}(\uv)}{N(I)} \le  C \Xnorm{\uv}{S(I)}^3.
  \end{eqnarray}

\subsection{Local well-posedness and extension of solutions}
Let us begin with the well-poesdness of \eqref{E:gNLS}.
\begin{theorem}[Local well-posedness]\label{thm.3.2}
  (gNLS) is locally well-posed in $\Hsp$. 
  Namely, for all $t_0 \in \R$, $\uv_0 \in \Hsp$, there exists $T=T(\norm{\uv_0}_{(H^1(\R^3))^N}) >0$  
      and a unique $H^1$-solution $\uv(t)$ to \eqref{E:gNLS} on $[t_0-T,\,t_0+T]$. 
      The solution continuously depends on the data.
\end{theorem}

\begin{theorem}[Conservation laws]\label{thm.3.3}
  Let $\uv$ be an $H^1$-solution to (gNLS) on $I$. 
  Then, $M(\uv(t))$, $E(\uv(t))$, and $P(\uv(t))$ do not depend on $t$.
\end{theorem}
We omit the detail of the proof. 
The assumptions \eqref{E:Hamiltonian} and \eqref{gauge condition} give us
the identities
\begin{align}
  & \Re \sum_{j=1}^{N} F_j \overline{\partial_t u_j} = \partial_t (\tfrac{1}{4} g(\uv)), \label{eq:1.1}\\
  & \sum_{j=1}^{N} \Im (F_j(\bm{z})\overline{z_j}) = 0,\ \ (\bm{z}=(z_1,z_2,\dots,z_N) \in \C^N), \label{eq:1.2}
\end{align}
respectively.
These identities play an essential role in proving the conservation laws. 

We have global bound in $S$-norm for small data.

\begin{proposition}\label{cor.3.8}
(1) There exists $\delta_1>0$ such that 
for any $\uo \in \Hsp$ and $t_0 \in \R$ satisfying  
$  \Xnorm{\Fl}{S(\R)} \le \delta_1$, 
there exists an $H^1$-solution $\uv$ to (gNLS) on $\R$, which satisfies
  \begin{align*}
    \Xnorm{\uv}{S(\R)} \le 2 \Xnorm{\Fl}{S(\R)}.
  \end{align*}
(2) There exists $\delta_2>0$ such that 
for any $\uo \in \Hsp$ satisfying 
$  \Xnorm{\uo}{(L^2(\R^3))^N} \Xnorm{\nabla \uo}{(L^2(\R^3))^N} \le \delta_2$,
there exists an $H^1$-solution $\uv$ to (gNLS) on $\R$, which satisfies
$  \Xnorm{\uv}{S(\R)} \le C \Xnorm{\uo}{(L^2(\R^3))^N} \Xnorm{\nabla \uo}{(L^2(\R^3))^N}$.
\end{proposition}

The first assertion follows from, for instance, the fixed point theorem in $S(\R)$
space. 
This is achieved by \eqref{SN esti.} and \eqref{NS esti.}.
The second follows from the first by the Strichartz estimate and the Sobolev embedding.

To discuss global existence/finite-time blowup,
let us make notation.
\begin{definition}[Maximum existence time]\label{def.3.9}
  Let $t_0 \in \R$, $\uo \in \Hsp$, $\uv$ be an $H^1$-sol to (gNLS). We define
  \begin{align*} 
  \TM :={}& \sup\{T \in \R \mid \uv \text{ can be extended to a solution on } [t_0,T] \}, \\
  \Tm :={}& \inf\{T \in \R \mid \uv \text{ can be extended to a solution on } [T,t_0] \}.
  \end{align*} 
  $\IM := (\Tm,\TM)$ is called a maximal interval of existence.
 A solution $\uv$ is said to be a maximal-lifespan solution
  if  it is a solution on $\IM$.
  If $\IM=\R$, $\uv$ is called a global solution.
\end{definition}
We have the standard blowup alternative.

\begin{proposition}[Blowup alternative]\label{prop.3.10}
  For $t_0 \in \R$, $\uo \in \Hsp$, let $\uv$ be a maximal-lifespan solution. 
  If $\TM < \infty$ then 
  \begin{align*}
    \lim_{t \nearrow \TM} \Xnorm{\uv(t)}{H^1} ={}& \infty,&\text{and}&
    &\Xnorm{\uv(t)}{S([t_0,\TM))} = \infty.
  \end{align*} 
  The same holds for the negative time direction. 
\end{proposition}

\subsection{On scattering}
Let us next
establish an equivalent characterization of the scattering. 

\begin{proposition}[Equivalent characterization of scattering]\label{prop.3.12}
  Let $\uv$ be an $H^1$-maximal lifespan to (gNLS). 
  $\uv(t)$ scatters in the positive direction if and only if  there exists $t_0 \in \IM$ such that
$      \Xnorm{\uv}{S( [t_0 , \TM))} < \infty.
$
  Similarly, $\uv(t)$ scatters in the negative direction if and only if  there exists $t_0 \in \IM$ such that
 $  
    \Xnorm{\uv}{S( (\Tm , t_0])} < \infty.
  $
\end{proposition}

  We show that, given $\uv_+$, 
  it is possible to construct a solution that satisfies $\TM = \infty$ and $\lim_{t \to \infty} U(-t)\uv(t) = \uv_+$ in $\Hsp$.
\begin{theorem}[Final state problem]\label{thm.3.13}
  Let $\uv _+ \in \Hsp$. Then, there exists $T=T(\Xnorm{\uv _+}{H^1})$ 
  and a unique $H^1$-solution to (gNLS) on $[T,\infty)$ such that
  $   U(-t)\uv(t) \longrightarrow \uv_+ \ \ (t \to \infty) \ \  \text{in}\ \  H^1.
$\end{theorem}

\subsection{Long time pertubation}\label{S:pd}
One way to investigate the scattering of  a solution $\uv$ is to compare 
it with a function $\utilde$ which has a finite $S$-norm,  almost solves (gNLS), and is close to $\uv$ at some time.
This technique is called the long time pertubation. 

\begin{theorem}[Long time pertubation]\label{thm.3.14}
  Let $t_0 \in \R$ and let $I$ be an interval such that $t_0 \in \bar I$. 
  Let $\uo \in \Hsp$. 
  Let $\utilde \in (C(I;H^1))^2$ be given. 
  Set $M := \Xnorm{\utilde}{S(I)} < \infty$ and $\bm{\Lambda} := \Equ{\utilde} + \bm{F}(\utilde) $. 
  Then, there exists $\ep_0 = \ep_0(M) > 0$ and $C=C(M) >0$ such that if  
  \begin{eqnarray*}
    \ep := \Xnorm{U(t-t_0)(\uo - \utilde(t_0))}{S(\R)} + \Xnorm{\bm \Lambda}{N(I)} \le \ep_0,
  \end{eqnarray*}
  then the $H^1$-solution $\uv$ to \eqref{E:NLS} with the initial condition $\uv(t_0)=\uv_0$
  exists on $I$ such that 
 $   \Xnorm{\uv - \tilde \uv}{S(I)} \le C \ep.
$
\end{theorem}

\section{Profile decomposition theorem}\label{S:pd}
In this section, we prove the profile decomposition of a  sequence bounded in $\Hsp$.
\begin{definition}[Frequency cutoff]\label{def.1.8}
  For a dyadic number $M \in 2^{\mathbb{Z}}$,  let
  $\chi_M(x) := \chi (x/M)$ be the Littlewood-Paley decomposition. Namely, 
   $\chi_1 \in C_0^{\infty}(\R)$ is chosen so that $\sum_{M\in 2^\Z} \chi_M(x)=1 $ for $\R^3 \setminus \{0\}$. 
  We define  $P_M =\chi_M(i \nabla):= \F^{-1}  \chi_M \F$,
  $\PN{\le} =\sum_{K\le M}P_K$, and
  $\PN{\ge} =\sum_{K \ge M}P_K$.
  Similarly, we define $\PN{<}$ and $\PN{>}$.
\end{definition}

For any $x,y \in \R^3$, define the translation operator by
\begin{align*}
  (T_y f)(x)\coloneqq f(x-y), \ \ T_{y}^{-1} \coloneqq T_{-y}.
\end{align*}
$T_y$ is a unitary group operator in $H^1$.

\begin{theorem}[Profile decomposition]\label{thm.4.7}
  For any bounded sequence $\br{\uv_n} \subset \Hsp$, 
 there exists subsequence of $n$, $J_0 \in \N_0 \cup \br{\infty}$, $\phiv_j \in \Hsp \setminus \br{0}$, 
  $\br{\yn{j}} \subset \R^3$, $\br{\tn{j}} \subset \R$, and $\br{\rv{j}} \subset \Hsp$ 
  such that the followings are true:
 \begin{enumerate}\renewcommand{\theenumi}{\roman{enumi}}
  \item For $J \in [1,J_0]$, 
  \begin{align*}
    \uv_n = \sum_{j=1}^{J} U(\tn{j})T_{\yn{j}} \phiv_j +\rv{J}.
  \end{align*}
  \item For $1 \le j_1 < j_2 \le J_0$, 
  $|\yn{j_1} -\yn{j_2}| + |\tn{j_1} - \tn{j_2}| \rightarrow \infty$
  as $n \to \infty$.  Further, for each $j\in [1,J_0]$, either $t_n^j\equiv 0$, $t_n^j\to\infty$ as $n\to\infty$, or $t_n^j\to-\infty$ as $n\to\infty$.
  \item With the convention $\rv{0} = \uv_n$, 
  \begin{align*}
    \UTinv{j}\rv{k} \rightharpoonup \begin{cases}
      0 & k \ge j \\
      \phiv_j & k< j 
    \end{cases}
    \ \ \ \text{in} \ \ \Hsp \ \ \ (n \to \infty).
  \end{align*}
  \item $\displaystyle\lim_{J \to J_0} \lims_{n \to \infty} 
  \Xnorm{U(\cdot) \rv{J}}{\Lsp{\infty}{\R}{(L^4)^N}} = 0$.
  \item For $A= \mathrm{Id}$ or $\nabla$, $\displaystyle\lim_{n \to \infty} \Xnorm{A\uv_n}{(L^2(\R^3))^N}^2 \ge \sum_{j=1}^{J_0} \Xnorm{A\phiv_j}{(L^2(\R^3))^N}^2.$
  \item Let $G(\uv) = \frac{1}{4} \int_{\R^3} g(\uv)dx$ and $\mathcal{J} := \{j \in [1,J_0] | t_n^j \equiv 0\}$.
  Then, $$\lim_{n \to \infty} G(\uv_n) = \sum_{j\in \mathcal{J}} G(\phiv_j).$$
   \end{enumerate}
\end{theorem}
\subsection{Characterization of the orthogonality}
\begin{proposition}[]\label{prop.4.7.2}
  Let $\br{a_n}$ and $\br{b_n}$ be sequences in $\R^3$, $\br{p_n}$ and $\br{q_n}$ be sequences in $\R$. 
  The following three are equivalent:
  \begin{enumerate}
\item  $  |a_n-b_n|+|p_n-q_n| \to \infty$ as $n \to \infty$;
\item  For any $\bm{\phi} \in \Hsp$,
$U(p_n-q_n)T_{a_n-b_n}\bm{\phi} \rightharpoonup 0$ in $\Hsp$ as $n \to \infty$;
\item  For any subsequence $\br{n_k}_{k=0}^\infty$,
 there exists a bounded sequence $\br{\bm{r}_k}$ in  $\Hsp$ and a subsequence of $k$ such that
  \begin{align*}
    (U(p_{n_k})T_{a_{n_k}})^{-1} \bm{r}_{k}  &{}\rightharpoonup \bm{\phi}, &
    (U(q_{n_k})T_{b_{n_k}})^{-1} \bm{r}_{k}  &{}\rightharpoonup 0
\end{align*} 
for some $\bm{\phi} \neq 0 $
along the subsequence.
  \end{enumerate}
\end{proposition}
\begin{proof}[Proof of Proposition \ref{prop.4.7.2}] 
``(1) $\Rightarrow$ (2)'' follows by the standard argument. 
``(2) $\Rightarrow$ (3)'' is obvious by taking 
$r_k= U(p_{n_k})T_{a_{n_k}} \bm{\phi}$ for some $\bm{\phi} \neq 0 $.

Let us prove ``(3) $\Rightarrow$ (1)'' by showing its contraposition.
  Let $\br{a_n}$, $\br{b_n}$, $\br{p_n}$, and $\br{q_n}$ be sequences such that (1) fails, i.e.,
  $\varliminf_{n\to\infty} (|a_n-b_n|+|p_n-q_n|)<\infty$. 
  By Boltzano-Weierstrass theorem, we obtain a subsequence $\{n_k\}_k$ such 
$  a_n-b_n \rightarrow a_\infty$ and
   $ p_n-q_n  \rightarrow p_\infty$ as $n\to\infty$,
for some $a \in \R^3$ and $ p \in \R$.
Then, for any sequence $\{ r_k \}_k \subset \Hsp$ satisfying
\begin{align*}
    (U(p_{n_k})T_{a_{n_k}})^{-1} \bm{r}_{k}  &{}\rightharpoonup \bm{\phi}, &
    (U(q_{n_k})T_{b_{n_k}})^{-1} \bm{r}_{k}  &{}\rightharpoonup 0
\end{align*}
along a subsequence of $k$, we have
\begin{align*}
	 (U(p_{n_k})T_{a_{n_k}})^{-1} \bm{r}_{k}
	 &{}=  (U(p_{n_k}-q_{n_k})T_{a_{n_k}-b_{n_k}})^{-1} \bm{r}_{k} 
	 ((U(q_{n_k})T_{b_{n_k}})^{-1} \bm{r}_{k}) \\
	 &{}\rightharpoonup   (U(p_\infty)T_{a_\infty})^{-1} 0
\end{align*}
as $k\to \infty$. This shows $\bm{\phi}=0$. (3) fails.
\end{proof} 

\subsection{Lieb's compactness theorem}

  \begin{proposition}[]\label{prop.4.7.1}
    If a sequence $\br{\uv_n}$ satisfies
    \begin{align*}
      \sup_{n \in \N}\Xnorm{\uv_n}{(H^1(\R^3))^N} &{}\le R, &
      \lims_{n \to \infty}\Xnorm{\uv_n}{L^{\infty}(\R;(L^{4})^N)} &{}\ge \ep
    \end{align*}
    for some $R>0$ and $\ep>0$,
    then there exist a subsequence of $n$  and parameters $y_n \in \R^3$ and $t_n \in \R$ such that 
    $(U(t_n)T_{y_n})^{-1}\uv_n$ converges weakly to  $\phiv$ in $\Hsp$.
    Further, there exists $\beta=\beta(\ep,R)>0$ such that $\Xnorm{\phiv}{H^1} \ge \beta$. 
  \end{proposition}
  \begin{proof}[Proof of Proposition \ref{prop.4.7.1}] 
  In this proof, we use the abbreviation $L^{p}X =L^p(\R;X^N)$.
    Let $M$ be a dyadic number to be defined later. 
    By Bernstein's estimate, we have
    \begin{align*}
      \Xnorm{\PN{>} U(\cdot) \uv_n}{\Lpq{\infty}{4}} 
      & \le C M^{-\frac{1}{4}} \Xnorm{|\nabla|^\frac{1}{4} \PN{>} U(\cdot) \uv_n}{L^{\infty}\dot{H}^\frac{3}{4}}
       \le C M^{-\frac{1}{4}} R.
    \end{align*}
    Taking $M$ so that $C M^{-\frac{1}{4}} R \le \frac{\ep}{2}$, we have
    \begin{align*}
      \lims_{n \to \infty} \Xnorm{\PN{\le} U(\cdot) \uv_n}{\Lpq{\infty}{4}} 
      &\ge \lims_{n \to \infty}  (\Xnorm{U(\cdot) \uv_n}{\Lpq{\infty}{4}} -  \Xnorm{\PN{>} U(\cdot) \uv_n}{\Lpq{\infty}{4}}) \\
      & \ge \tfrac{\ep}{2}.
    \end{align*}
  On the other hand, it follows from H\"older inequality that
    \begin{align*}
      \Xnorm{\PN{>} U(\cdot) \uv_n}{\Lpq{\infty}{4}}
      \le R^\frac{1}{2} \Xnorm{\PN{>} U(\cdot) \uv_n}{\Lpq{\infty}{\infty}}^{\frac12}.
    \end{align*}
    Therefore, 
$       \lims_{n \to \infty} \Xnorm{\PN{\le} U(\cdot) \uv_n}{\Lpq{\infty}{\infty}} 
      \ge \tfrac{\ep ^2}{4R}$.
    Thus, there exist  $\br{y_n} \subset \R^3$ and $\br{t_n} \subset \R$  such that
 $     | \PN{\le} U(t_n) \uv_n(y_n) | \ge \tfrac{\ep ^2}{8R}
$    along a subsequence.
	By taking a subsequence if necessary, we have 
    \begin{align*}
      (U(t_n) \uv_n(y_n))^{-1} \uv_n \rightharpoonup \phiv \ \ \text{in}\ \ \Hsp
    \end{align*}
    as $n\to\infty$.
    By the definition of $\PN{\le}$, we have
    \begin{align*}
      (\PN{\le} U^{-1}(t_n)\uv_n)(y_n) 
      &{}= \Finverse{{\chi}({\cdot}/M)\F({U(t_n)^{-1}\uv_n})}(y_n) 
      \\
      &{}= C M^3 \int_{\R^3}  m(-Mx)(U(t_n)T_{y_n})^{-1}\uv_n (x) dx,
    \end{align*}
    where $m= \mathcal{F}^{-1} {\chi}$.
 Hence, we obtain
    \begin{align*}
      (\PN{\le} U^{-1}(t_n)\uv_n)(y_n) 
      \to C N^3\int_{\R^3} m(-Nx) \phiv(x) dx
    \end{align*}
    as $n\to\infty$.
    Thus, we have
    \begin{align*}
      \tfrac{\ep ^2}{8R} & \le \left| CM^3 \int_{\R^3} m(-Mx) \phiv(x) dx \right| 
       \le  CM^\frac{3}{2} \Xnorm{m}{L^2}\Xnorm{\phiv}{H^1},
    \end{align*}
which reads as
 $     \Xnorm{\phiv}{H^1}\ge \tfrac{\ep^2}{8CRM^3 \Xnorm{m}{L^2}} 
      \eqqcolon \beta(\ep,R) >0.$
\end{proof}

\subsection{Proof of the theorem}
\begin{proof}[Proof of Theorem \ref{thm.4.7}] 
In this proof, we take subsequence of $n$ at most countably many times.
We denote them all by $n$, for simplicity. 
Moreover, we denote $\norm{\cdot}:=\norm{\cdot}_{(L^2(\R^3))^N}$. and $\inner{\cdot}{\cdot}:=\inner{\cdot}{\cdot}_{(L^2(\R^3))^N}$.

  Let $\br{\uv_n}$ be a bounded in $\Hsp$. By extracting a subsequence if necessary, we assume that $\lim_{n\to\infty} \norm{\uv}_{\Hsp}$ exists.
  We define 
  $  \nu (\br{\uv_n})$ as a set of functions $  \phiv \in \Hsp$ such that
  there exist a subsequence  and suquences $\br {y_n} \subset \R^3$ and $\br{t_n}\subset \R$
  such that
  $
  (U(t_n)T_{y_n})^{-1} \uv_n \rightharpoonup \phiv
  $
 holds as $n\to \infty$. Further, we let
  \begin{align*}
    \eta (\br{\uv_n}) & \coloneqq \sup_{\phiv \in \nu (\br{\uv_n})} \Xnorm{\phiv}{\Hsp}.
  \end{align*}
{\bf Step 1}.
  We find profiles by utilizing the definition of $\nu (\br{\uv_n})$. 
  When $\eta(\br{\uv_n}) = 0$, 
  we have $\sup_{\phiv \in \nu (\br{\uv_n})} \Xnorm{\phiv}{\Hsp} = 0$. 
  Thus, $\nu (\br{\uv_n})=\{0\}$.  
  So, by taking $J_0=0$, $\rv{0}=\uv_n$, we have (i), (ii), (iii) and (v). 
  Next, we prove (iv). 
  Suppose that $\lims_{n \to \infty} \Xnorm{U(\cdot)\uv_n}{L^{\infty}(L^{4})^N} >0$. 
  By applying Proposition $\ref{prop.4.7.1}$, we find a nonzero element in $ \nu (\br{\uv_n})$, a contradiction.
So the result holds. 
  Therefore, we assume $\eta(\br{\uv_n}) > 0$ in what follows. 
 
 Since $\eta (\br{\uv_n})>0$, there exists $\phiv_1 \in \nu (\br{\uv_n})$ such that
$    \Xnorm{\phiv_1}{H^1} \ge \frac{1}{2} \eta(\br{\uv_n})>0$. In particular, $\phi_1 \neq0$.
  By the definition of $\nu (\br{\uv_n})$, there exist a subsequence, $\yn{1} \subset \R^3$, and $\tn{1} \subset \R$ 
  such that
  \begin{align*}
    \UTinv{1}\uv_n \rightharpoonup \phiv_1\ \ \ (n \to \infty).
  \end{align*} 
  In the following, we determine $\phiv_j$, $\yn{j}$, $\tn{j}$ and $J_0$ inductively. 

  If $\eta(\br{\rv{j}}) = 0$, terminate the operation with $J_0=j$. 
  
  If $\eta(\br{\rv{j}}) > 0$, in the same way as in $j=1$ one finds $\phiv_{j+1} \in \nu (\br{\rv{j}})$ such that
  \begin{align} \label{eq:4.7.1}
    \Xnorm{\phiv_{j+1}}{\Hsp} \ge \tfrac{1}{2} \eta(\br{\rv{j}}),
  \end{align} 
  and there exist a subsequence, $\yn{j+1} \subset \R^3$, $\tn{j+1} \subset \R$ such that
  \begin{align} \label{eq:4.7.2}
    \UTinv{j+1}\rv{j} \rightharpoonup \phiv_{j+1}\ \ \ (n \to \infty).
  \end{align} 
  In addition, we define  
  \begin{align} \label{eq:4.7.3}
    \rv{j+1}\coloneqq \rv{j}-\UTinv{j+1} \phiv_{j+1},
  \end{align} 
  Then, by $\eqref{eq:4.7.2}$, we have  
  \begin{align} \label{eq:4.7.4}
    \UTinv{j+1}\rv{j+1} \rightharpoonup 0\ \ \ (n \to \infty).
  \end{align} 
  Then, we proceed to the next $j$. 

  If the procedure does not finish in a finite time, we let $J_0=\infty$. 

{\bf Step 2.}
  We prove (i), (iv) and (v). 
  Firstly, (i) follows from \eqref{eq:4.7.3}.
  Let us prove (v). Let $A$ be an identity mapping or $A=\nabla$. By \eqref{eq:4.7.3},
  \begin{align*}
    \norm{A \rv{j}}^2 
    &{}= \norm{A \rv{j+1}}^2 + \norm{A\phiv_{j+1}}^2 \\
    &\qquad + 2 \re \inner{ A\rv{j+1}}{(\UT{j+1})A \phiv_{j+1} }.
  \end{align*} 
  The third term in the right hand side tends to zero as $n\to \infty$ because
  \begin{align*}
    \UTinv{j+1} \rv{j+1} &= \UTinv{j+1} (\rv{j} - \UT{j+1} \phiv_{j+1}) \\
    &=\UTinv{j+1} \rv{j} - \phiv_{j+1} 
     \rightharpoonup 0
  \end{align*} 
  as $n \to \infty$.
  Recalling that $\norm{\uv}_{\Hsp}$ is convergent, one deduces that
  \begin{align}\label{eq:4.7.6}
 \lim_{n\to\infty}   \norm{A \rv{j}}^2 = \lim_{n\to\infty}\norm{A \rv{j+1}}^2 + \norm{A\phiv_{j+1}}^2 
  \end{align} 
  for all $J < J_0$.
Hence, by summing this up, one obtains
  \begin{align}\label{eq:4.7.7}
    \lim_{n\to \infty}\norm{A \uv_n}^2 &= \lim_{n\to \infty} \norm{A \rv{j}}^2 + \sum_{j=1}^J \norm{A\phiv_{j}}^2 
     \ge \sum_{j=1}^J \norm{A\phiv_{j}}^2.
  \end{align} 
Thus, (v) follows by taking supremum with respect to $J$.

 Let us move on to the proof of (iv). 
  If $J_0 < \infty$, then $\eta (\br{\rv{J_0}})=0$. 
It follows from Proposition \ref{prop.4.7.1} that
 $   \lims_{n \to \infty} \Xnorm{U(\cdot)\rv{J_0}}{L^{\infty}(L^{4})^N} =0$.
 Suppose that $J_0 = \infty$. 
 Suppose for contradiction that 
  \begin{align*}
    \exists \ep >0 ,\, \exists \br{J_k}_{k \in \N},\, J_k \to \infty\,(k \to \infty),\,\forall k,\, \lims_{n\to\infty} \Xnorm{U(\cdot)\rv{J_k}}{L^{\infty}(L^{4})^N}  \ge \ep.
  \end{align*} 
  By $\eqref{eq:4.7.7}$, we have
  \begin{align*}
    \lims_{n \to \infty} \Xnorm{U(\cdot)\rv{J_k}}{(H^1)^N}^2 = \lims_{n \to \infty} \Xnorm{\rv{J_k}}{(H^1)^N}^2
    \le \lim_{n \to \infty} \Xnorm{\uv_n}{(H^1)^N}^2.
  \end{align*} 
  Set $\lim_{n \to \infty} \Xnorm{\uv_n}{(H^1)^N}^2 \eqqcolon R$. Then, 
  we have $ \eta (\br{\rv{J_k}})\ge \beta (\ep,R)>0$, where $\beta$ is given in Proposition \ref{prop.4.7.1}.
  However, by (v), we obtain $\sum_{j=1}^{\infty} \Xnorm{\phiv_j}{(H^1)^N} < \infty$. 
  Thus, by the definition of $\phiv_{J_k+1}$, we see that
$    \eta (\br{\rv{J_k}}) \le 2 \Xnorm{\phiv_{J_k+1}}{(H^1)^N} \rightarrow 0$ as $j\to \infty,$ 
  which is a contradiction. Hence, (iv) holds.
  
  {\bf Step 3}. 
  Let us prove (ii) and (iii). We consider the case $J_0=\infty$. 
  The latter half of (ii) follows by modifying the definition of $t_n^j$ and $\bm{\phi}_j$.
  We prove the former half of (ii) by induction on the value of $j_2-j_1$. 
By \eqref{eq:4.7.2} and \eqref{eq:4.7.4}, 
  \begin{align*}
    \UTinv{j+1}\rv{j} &\rightharpoonup \phiv_{j+1} \neq 0, &
    \UTinv{j}\rv{j} &\rightharpoonup 0
  \end{align*} 
  as $n\to\infty$.
Hence, we see from ``(3)$\Rightarrow$(1)'' of Proposition \ref{prop.4.7.2}
that (ii) is true if $j_2-j_1=1$.

Let $k_0\ge1$ and
suppose that (ii) is true if $j_2-j_1\le k_0$.
By (i),
\[
	\rv{j} = \sum_{m=1}^{k_0}\UT{j+m} \phiv_{j+m} + \rv{j+k_0+1}.
\]
Hence, by combining the assumption of the induction, ``(1)$\Rightarrow$(2)'' of Proposition \ref{prop.4.7.2}, and
\eqref{eq:4.7.2} with $j=j+k_0$, one sees that
\[
    \UTinv{j+k_0+1}\rv{j} \rightharpoonup \phiv_{j+k_0+1} \neq 0.
\]
Recalling that $    \UTinv{j}\rv{j} \rightharpoonup 0
$ as $n\to\infty$, one sees from ``(3)$\Rightarrow$(1)'' of Proposition \ref{prop.4.7.2}
that (ii) is true when $j_2-j_1=k_0+1$. By induction, we have (ii).
(iii) is a consequence of (i), (ii), and \eqref{eq:4.7.2}.

  Finally, let us prove (vi). Let $J \le J_0$ be finite. By homogeneity of $G$, we have
  \begin{align*}
    \left|
      G(\uv_n) - G ( \sum_{j=1}^{J} U(\tn{j}) T_{\yn{j}} \phiv_j  )
    \right|
   ={}& \left| - 
    \int_{0}^{1} \partial_{\theta} G(\uv_n - \theta \rv{J}) d \theta 
  \right|\\
  \lesssim{}& (\Xnorm{\uv_n}{L^4} + \Xnorm{\rv{J}}{L^4} )^3 \Xnorm{\rv{J}}{L^4}.
  \end{align*}
  In addition, by the normalization rule of $t_n^j$ in (ii), dispersive estimate, and
  the mutual orthogonality of $\br{\yn{j}}$, we have
  \begin{align*}
    \lim_{n \to \infty} G ( \sum_{j=1}^{J} U(\tn{j}) T_{\yn{j}} \phiv_j  ) = \sum_{j\in \mathcal{J} \cap [1,J]} G(\phiv_j).
  \end{align*}
  Combining those equations, Galiardo-Nirenberg inequality, and \eqref{eq:4.7.7}, 
  \begin{align*}
    \lims_{n \to \infty} \left| G(\uv_n) - \sum_{j=1}^{J} G(\phiv_j) \right|
    \lesssim (\lim_{n \to \infty} \Xnorm{\uv_n}{\Hsp} ) ^{3}
    \lims_{n \to \infty} \Xnorm{\rv{J}}{L^4} \to 0
  \end{align*}
 as $J \to J_0$. Thus (vi) holds. 
\end{proof}

\section{Variational characterization of the ground state}\label{S:gs}
In this section, we collect properties of the ground states shown in \cite{Masaki} without proof.

\begin{definition}[]\label{def.5.1}
We define the following functionals:
\begin{align*} 
 H(\uv) :={}& \frac{1}{2} \Xnorm{\nabla \uv}{(L^2(\R^3)^N)}^2, &
  G(\uv) :={}& \frac{1}{4} \int_{\R^3} g(\uv)dx, \\
  K(\uv) :={}& 2 H(\uv) - 3G(\uv), &
  S_{\omega}(\uv) :={}& E(\uv) + \omega M(\uv).
\end{align*}
\end{definition}

\begin{definition}[]\label{def.5.2}
  Let $\omega> 0$. For any $\uv \in \Hsp$, we define
  \begin{align*} 
    &\mathcal{A}_{\omega} := \br{\Qv \in \Hsp \mid \Qv \neq 0,\, S'_{\omega}(\Qv)=0}, \\
    &\mathcal{G}_{\omega} :=\br{\Qv \in \mathcal{A}_{\omega} 
    \mid S_{\omega}(\Qv)= \inf_{\bm{\psi}\in \mathcal{A}_{\omega} }  S_{\omega}(\bm{\psi})}.
  \end{align*}
  In addition, we define $\mathcal{A} := \bigcup_{\omega >0}\mathcal{A}_{\omega}$ and 
  $\mathcal{G} := \bigcup_{\omega >0}\mathcal{G}_{\omega}$. 
  \end{definition}
  
  \begin{remark}
 Note that $\uv \in \Hsp$ is a solution to \eqref{E:gE} if and only if
   $S'_{\omega}(\uv)=0$. 
    Therefore, $\mathcal{A}_{\omega} $ denotes the set of nontrivial solutions to \eqref{E:gE} (bound states).
	The elements in $\mathcal{G}_{\omega}$ are ground state. 
  \end{remark}
Then, as mentioned in Section 1, the following holds:
\begin{theorem}[\cite{Masaki}]
For all $\omega>0$, $\mathcal{G}_\omega$ is given as in \eqref{E:Gomega}.
\end{theorem}

Further, since the cubic nonlinearity is mass-supercritical in three space dimensions. There are the potential-well structure. 
We have the following variational characterization of the ground states.
\begin{theorem}[\cite{Masaki}*{Theorem 5.1}]\label{thm.5.10}
  Let $\Qv \in \mathcal{G} $. For $\delta \in (0,1)$, 
  there exists $\tilde{\delta}=\tilde{\delta}(\delta) >0$ 
  such that if $\uv \in \Hsp \setminus \{0\}$ satisfies
  \begin{align*}
    E(\uv)M(\uv) \le (1-\delta) E(\Qv)M(\Qv)
  \end{align*}
  then  $K(\uv) \neq 0$ holds and the followings are true:
  \begin{enumerate}
  \item If $K(\uv) > 0 \ \ $ then $\ \ H(\uv) M(\uv) < (1-\delta) H(\Qv)M(\Qv)$ and 
  $K(\uv) \ge \tilde{\delta} H(\uv)$.
  \item
   If $K(\uv) < 0 \ \ $ then $\ \ H(\uv) M(\uv) >  H(\Qv)M(\Qv)$ and 
  $K(\uv) M(\uv) \le - \tilde{\delta}  H(\Qv)M(\Qv)$.
  \end{enumerate}
\end{theorem}
By combining  the characterization and the analysis in Section \ref{S:lwp}, we have the following global existence 
result. 
\begin{corollary}[Global well-posedness below ground states]\label{cor.5.11}
  Let $\uo \in \Hsp \setminus \{0\}$ satisfy
  \begin{align*}
    E(\uo)M(\uo) \le (1-\delta) E(\Qv)M(\Qv),
  \end{align*}
  where some $\delta \in (0,1)$.  Then, \eqref{condition2} is equivalent to
  $K(\uo) >0$.
  Moreover, if $K(\uo) >0$ is satisfied then $\uv(t)$ is a  global  $H^1$-solution to \eqref{E:gNLS} satisfying 
  \begin{align*}
    & \sup_{t \in \R} H(\uv(t)) M(\uv(t)) 
    \le (1-\delta) H(\Qv)M(\Qv) ,\\
    & \inf_{t \in \R} K(\uv(t)) 
    \ge \tfrac{\tilde{\delta}}{2} \inf_{t \in \R} \Xnorm{\nabla \uv}{(L^2)^N}^2, 
  \end{align*}
  where $\tilde{\delta}$ is as in Theorem $\ref{thm.5.10}$. 
\end{corollary}
%

\section{Proof of the main theorem}\label{S:proof}

\subsection{Reformulation as a variational problem}
By means of Corollary \ref{cor.5.11}, Theorem \ref{main} is formulated as follows.
\begin{theorem}[]\label{thm.6.1}
  If $\uo \in \Hsp$ satisfy \eqref{condition1} and  $K(\uo) > 0 $.
 The global solution $\uv(t)$ given in Corollary \ref{cor.5.11} satisfies
$    \Xnorm{\uv}{S(\R)} < \infty.
  $
\end{theorem}
\begin{definition}
For $a>0$, we define 
\[
	L(a) := \sup \Xnorm{\uv}{S(\IM)} \quad \text{and} \quad   A_c := \sup \{  a \mid \, L(a) < \infty \},
\] where
the first supremum is taken over all maximal-lifespan $H^1$-solution $ \uv(t)$ satisfying
$M(\uv)E(\uv) \le a$ and $K(\uv(t_0)) \ge 0$ at some $t_0 \in I_{\max}$.
\end{definition}
Let us collect properties of $L$ and $A_c$.
\begin{proposition}[]\label{prop.6.2}
Pick $\Qv \in \mathcal{G}$.
\begin{enumerate}
 \item  $L\colon [0, \infty) \rightarrow [0,\infty]$ is nondecreasing.
 \item  $L(M(\Qv)E(\Qv)) = \infty$.
 \item  $0 <A_c \le M(\Qv)E(\Qv)$.
 \item  $L$ is continuous on $[0,A_c)$. 
 \end{enumerate}
\end{proposition}
\begin{proof}[Proof of Proposition \ref{prop.6.2}] 
  (1) is obvious by definition. 
  (2) is also immediate by the fact that ground states are included in the spremum in the definition  of $L(M(\Qv)E(\Qv))$
  and that a ground state is a non-scattering solution.
  Then, the second inequality of (3) follows from (1) and (2).
  The first inequality follows from (2) of Corollary \ref{cor.3.8} and Corollary \ref{cor.5.11}.
(4) is shown by Theorem \ref{thm.3.14}. We omit the detail.
\end{proof}

Then, one sees that Theorem \ref{thm.6.1} follows from the following theorem.

\begin{theorem}[]\label{thm.6.3}
  $A_c=M(\Qv)E(\Qv)$ holds. 
\end{theorem}
In the sequel, we prove Theorem $\ref{thm.6.3}$ by contradiction. To this end, we suppose that this fails, i.e.,
\begin{align}\label{eq:6.1}
  A_c<M(\Qv)E(\Qv).
\end{align}

\subsection{Existence of a critical element}

We first show that, under the assumption \eqref{eq:6.1}, there exists an optimizer to $A_c$.
The optimizer is the solution which is often called a \emph{critical element} or \emph{a minimal blow-up solution.}
\begin{theorem}[]\label{thm.6.4}
  Under the assumption of $\eqref{eq:6.1}$, 
  there exists an $H^1$-global solution $\uc(t)$ to \eqref{E:gNLS} such that 
  $\Xnorm{\uc}{S([0,\infty))} = \Xnorm{\uc}{S((-\infty,0])} = \infty$, $\sup_{t} M(\uc)H(\uc) < M(\Qv)H(\Qv),$
  and
$M(\uc)E(\uc) = A_c$.
\end{theorem}
This theorem is shown by applying the following convergence result to an optimizing sequence to $A_c$.

  \begin{proposition}[key convergence result]\label{prop.6.5}
    Let $\br{\wv_n}_n$ be a sequence of $H^1$-global solutions to \eqref{E:gNLS} satisfying the followings:   
    \begin{align}
      & M(\wv_n) E(\wv_n) \le A_c, \rightarrow A_c\ \ \ (n \to \infty) \label{eq:6.6a}
      ,\\
      & \sup_{n,t} M(\wv_n) H(\wv_n) < M(\Qv)H(\Qv) \label{eq:6.6b}
      ,\\
      & \sup_{n}\Xnorm{\wv_n (0)}{H^1} < \infty \label{eq:6.6c}
      ,\\
      & \Xnorm{\wv_n}{S([0,\infty))} \rightarrow \infty
      ,\, \Xnorm{\wv_n}{S((-\infty,0])} \rightarrow \infty \ \ \ (n \to \infty) \label{eq:6.6d}
      .
    \end{align}
    Then there exist a subsequence of $n$, $\br{y_n} \subset \R^3$, and $\wv_{\infty , 0} \in \Hsp$ 
    such that 
    \begin{align}\label{eq:6.7o}
      T_{y_n}^{-1}\wv_n(0)\rightarrow \wv_{\infty , 0} \ \ \ \text{in} \ \ \Hsp
    \end{align}
    as $n \to \infty$.
    In addition, a solution $\wv_{\infty}$ to \eqref{E:gNLS} with the data
  $\wv_{\infty}(0)=\wv_{\infty , 0}$ exists globally and satisfies the followings:
    \begin{align}
      & M(\wv_{\infty}) E(\wv_{\infty}) = A_c \label{eq:6.7a}
      ,\\
      & \sup_{t} M(\wv_{\infty}) H(\wv_{\infty}) < M(\Qv)H(\Qv) \label{eq:6.7b}
      ,\\
      & \sup_{t}\Xnorm{\wv_{\infty} (t)}{H^1} < \infty \label{eq:6.7c}
      ,\\
      & \Xnorm{\wv_{\infty}}{S([0,\infty))}=\Xnorm{\wv_{\infty}}{S((-\infty,0])} = \infty \label{eq:6.7d}
      .
    \end{align}
  \end{proposition}

  \begin{proof}[Proof of Proposition \ref{prop.6.5}] 
    The proof is similar to \cite{Holmer-Roudenko}*{Proposition 5.4}.
    By assumption $\eqref{eq:6.6c}$, the sequence $\br{\wzero}$ is bounded  in $ \Hsp$. 
  We apply Theorem $\ref{thm.4.7}$ to the sequence. 
  Then,  there exist a subsequence for $n$, $J_{0} \in \N_0 \cup \br{\infty}$, 
    $\phiv_{j}=(\phi_{j,1},\phi_{j,2},\dots,\phi_{j,N}) \in \Hsp \setminus \br{0}$, 
    $\br{\yn{j}} \subset \R^3$, $\br{\tn{j}} \subset \R$, and $\br{\rv{j}} \subset \Hsp$  
    such that 
    \begin{align}
      \wzero = \sum_{j=1}^{J} U(\tn{j})T_{\yn{j}} \phiv_j +\rv{J} \label{eq:6.8a}
    \end{align}
    for every $1 \le j \le J_0$ and $n \ge 1$, 
    \begin{align}
      |\yn{j_1} -\yn{j_2}| + |\tn{j_1} - \tn{j_2}| \rightarrow \infty
      \ \ \ (n \to \infty)
      \label{eq:6.8b}
    \end{align}
    for every $1 \le j_1 < j_2 \le J_0$, and
    \begin{align}
      &\lims_{n \to \infty} \Xnorm{U(\cdot) \rv{J}}{\Lsp{\infty}{\R}{(L^4)^N} \cap S(\R)}
      \rightarrow 0 \ \ \ (J \to J_0).
      \label{eq:6.8c}
    \end{align}
One also has
    \begin{align}
      &\lim_{n \to \infty} M(\wzero) \ge \sum_{j=1}^{J_0} M(\phiv_j), \label{eq:6.8d} \\
      &\lim_{n \to \infty} H(\wzero) \ge \sum_{j=1}^{J_0} H(\phiv_j), \label{eq:6.8e} \\
      &\lim_{n \to \infty} G(\wzero) = \sum_{
        \substack{j \in [1,J_0] \\ \tn{j} \equiv 0}} G(\phiv_j).  \label{eq:6.8f}
    \end{align}
    Note that we can assume  
    $y_n$ satisfying  $\yn{j} \equiv 0$ or $|\yn{j}| \rightarrow \infty$ as $n \to \infty$ 
    and $t_n$ satisfying $\tn{j} \equiv 0$ or $\tn{j} \rightarrow \pm \infty$ as $n \to \infty$
    for every $j$. 
 
    We first show the following.
    \begin{lemma}\label{lem.6.6}
    One has
  $       \sum_{j=1}^{J_0} M(\phiv_j) H(\phiv_j) <  M(\Qv)H(\Qv). $ 
    Further, if $t_n^j \equiv 0$ then $K(\phiv_j) >0$ and $E(\phiv_j) >0$.
    \end{lemma}
    \begin{proof}[Proof of Lemma \ref{lem.6.6}] 
  The first   inequality  follows from $\eqref{eq:6.6b}$, $\eqref{eq:6.8d}$ and $\eqref{eq:6.8e}$. 
    The latter half follows from the inequality and the variational characterization of the ground state
    (See \cite[Lemma 5.4]{Masaki}).
    \end{proof}

    \begin{lemma}[]\label{lem.6.7}
      Define $\tilde{E}(\phiv_j) := \lim_{n \to \infty} E(U(\tn{j})T_{\yn{j}} \phiv_j )$. 
      Then, $\tilde{E}(\phiv_j) =E(\phiv_j)$ if $\tn{j}\equiv0$ and
      $ \tilde{E}(\phiv_j) = H(\phiv_j)$ if $\tn{j}\to \pm \infty$ as $n\to\infty$.
	In particular, $\tilde{E}(\phiv_j)>0$ for $j \in [1,J_0]$.
      In addition, the followings are true:
      \begin{align}
        &\lim_{n \to \infty} E(\wv_n) \ge \sum_{j=1}^{J_0}\tilde{E}(\phiv_j), \label{eq:6.12} \\
        &\sum_{j=1}^{J_0} M(\phiv_j) \tilde{E}(\phiv_j) \le A_c. \label{eq:6.13} 
      \end{align}
    \end{lemma}
    \begin{proof}[Proof of Lemma \ref{lem.6.7}] 
The former half is an immediate consequence of the dispersive estimate for $U(\cdot)$.
      Next, \eqref{eq:6.12} follows from $\eqref{eq:6.8c}$ and $\eqref{eq:6.8f}$. Indeed,
      \begin{align*}
        \lim_{n \to \infty} E(\wv_n) & \ge \lim_{n \to \infty} (H(\wzero) -G(\wzero)) \\
        &\ge \sum_{j=1}^{J_0}H(\phiv_j)  
        - \sum_{\substack{j \in [1,J_0] \\ \tn{j} \equiv 0}} G(\phiv_j) 
        =\sum_{j=1}^{J_0} \tilde{E}(\phiv_j).
      \end{align*}
      Let us finally consider \eqref{eq:6.13}. If $j$ satisfies $\tn{j} \to \pm\infty$ as $n\to\infty$,  
      then $\tilde{E}(\phiv_j) = H(\phiv_j) >0$.  
      If $j$ satisfies $\tn{j} \equiv 0$ , then $\tilde{E}(\phiv_j) = E(\phiv_j) >0$ by Lemma $\ref{lem.6.6}$. 
      Therefore, For any $j \in [1,J_0]$, $\tilde{E}(\phiv_j)>0$ follows.
	Then, \eqref{eq:6.13} then follows from $\eqref{eq:6.6a}$, $\eqref{eq:6.8d}$ and $\eqref{eq:6.12}$. 
    \end{proof}
    \begin{lemma}[Only one profile]\label{lem.6.8}
      $J_0 =1$.
    \end{lemma}
    \begin{proof}[Proof of Lemma \ref{lem.6.8}] 
     Let us prove this by contradiction. Suppose $J_0=0$. 
     Then  
$     \Xnorm{U(t)\wzero}{S(\R)} \rightarrow 0$ as $n \to \infty
 $  by $\eqref{eq:6.8a}$ and $\eqref{eq:6.8c}$. 
     By Corollary $\ref{cor.3.8}$ (1), one has 
$
      \Xnorm{\wv}{S(\R)} \le 2 \Xnorm{U(t)\wzero}{S(\R)} \to 0
$ 
    as $n \to \infty.$
    This contradicts $\eqref{eq:6.8d}$. Next, suppose $J_0 \ge 2$. 
    For each $j\in [0,J_0]$, we define a nonlinear profile  $\Pv_j(t)$ as follows:
    If $\tn{j} \equiv 0$, then we define $\Pv_j(t)$  as a maximal-lifespan solution  to \eqref{E:gNLS} by
    \begin{align}\label{eq:6.14a}
      \Pv_j(0) = \phiv_j.
    \end{align}
    Then, $E(\Pv_j)=E(\phiv_j)= \tilde{E}(\phiv_j)$. 
    If $\tn{j} \rightarrow \pm \infty$, 
    we define $\Pv_j(t)$ as a maximal-lifespan solution to \eqref{E:gNLS} satisfying
    \begin{align}\label{eq:6.14b}
      \lim_{n \to \pm \infty} U(-t) \Pv_j(t) = \phiv_j\ \ \ \text{in}\ \ \Hsp
    \end{align}
    This solution is given by Theorem $\ref{thm.3.13}$, 
   Moreover, we define 
    \begin{align}\label{eq:6.15}
      \tilde{\wv}_n^{J} (t)= \sum_{j=1}^{J}T_{\yn{j}} \Pv_j(t+\tn{j}) + U(t)\rv{J}. 
    \end{align}
    
   To estimate $\Xnorm{\wv_n}{S(\R)}$, we apply
  Theorem $\ref{thm.3.14}$ with taking $\tilde{\wv}_n^{J}$ as an approximate solution on $\R$.
   To this end, let us verify that the assumption of Theorem $\ref{thm.3.14}$ is satisfied. 
    By $\eqref{eq:6.13}$, for any $j \in [1,J_0]$, we have $M(\Pv_j)E(\Pv_j) < A_c$. 
    Therefore, we have 
    \begin{align*}
      \Xnorm{\Pv_j}{S(\R)} \le L(M(\Pv_j)E(\Pv_j)) < \infty.
    \end{align*}
    Thus, 
    \begin{align}\label{eq:6.16}
      \Xnorm{ \tilde{\wv}_n^{J}}{S(\R)} 
      &\le \sum_{j=1}^{J} \Xnorm{\Pv_j}{S(\R)} +\Xnorm{U(t)\rv{j}}{S(\R)} 
      < \infty 
    \end{align}
    follows. Next, we prove
    \begin{align}\label{eq:6.17}
      \sup_{J}\lims_{n \to \infty} \Xnorm{ \tilde{\wv}_n^{J}}{S(\R)} \le C_0.
    \end{align}
    It is obvious when $J_0$ is finite.
    Let us consider the case $J_0= \infty$. 
    By Corollary $\ref{cor.3.8}$ (2) and 
     $\sum_{j=1}^{\infty} \Xnorm{\phiv_j}{\Hsp}^2 < \infty$,  one has
	\begin{equation}\label{eq:6.17.5}
    \begin{aligned}
      \lims_{n \to \infty} \Xnorm{\sum_{j=J'}^{\infty} T_{\yn{j}} \Pv_j(t+\tn{j}) }{S(\R)} 
      & \le \lims_{n \to \infty} \sum_{j=J'}^{\infty}  (\Xnorm{T_{\yn{j}} \Pv_j(t+\tn{j}) }{S(\R)}^4 )^{\frac{1}{4}} \\
      & \le \sum_{j=J'}^{\infty}  (C \Xnorm{\phiv_j }{\Hsp}^4 )^{\frac{1}{4}} \le C \delta
    \end{aligned}
    \end{equation}
    for large $J'$.
Hence, together with \eqref{eq:6.8c}, we obtain  \eqref{eq:6.17}.

    Next, we evaluate the difference at $t=0$. By definition of the nonlinear profiles, one has 
    \begin{align} \label{eq:6.18}
      \Xnorm{U(t) (\wzero - \tilde{\wv}_n^{J}(0))}{S(\R)}
      & \le \sum_{j=1}^{J} \Xnorm{U(t) (U(\tn{j})T_{\yn{j}}  \phiv_j - T_{\yn{j}} \Pv_j(\tn{j}) )}{S(\R)} \notag \\
      & \le C \sum _{j=1} ^J \Xnorm{U(\tn{j}) \phiv_j -\Pv_j(\tn{j}) }{\Hsp}\to 0
    \end{align}
    as $n\to\infty$.

    Finally, we evaluate $\bm{\Lambda}_n^J := \Equ{\tilde{\wv}_n^{J}} + \bm{F}(\tilde{\wv}_n^{J})$. 
    We have
    \begin{align*}
      \bm{\Lambda}_n^J 
      & = \left(\bm{F}(\sum_{j=1}^{J} T_{\yn{j}} \Pv_j(t+\tn{j}) + U(t) \rv{j}) 
      - \bm{F}(\sum_{j=1}^{J} T_{\yn{j}} \Pv_j(t+\tn{j})) \right)\\
      &\ \ \ +\left(
        \bm{F}(\sum_{j=1}^{J} T_{\yn{j}} \Pv_j(t+\tn{j}))
        - \sum_{j=1}^{J} \bm{F} (T_{\yn{j}} \Pv_j(t+\tn{j}))\right) \\
        &=: I_1 + I_2.
    \end{align*}
   Let us evaluate $I_1$ and $I_2$. From H\"older's inequality, we obtain 
    \begin{align*}
      \Xnorm{I_1}{N(\R)} \le C \Xnorm{U(t) \rv{j}}{S(\R)}( 
        \Xnorm{\sum_{j=1}^{J} T_{\yn{j}} \Pv_j(t+\tn{j})}{S(\R)}^2 + \Xnorm{U(t) \rv{j})}{S(\R)}^2
    ).
    \end{align*}
    Thus, \eqref{eq:6.8c} and \eqref{eq:6.17.5} give us
$      \lim_{n \to \infty}\Xnorm{I_1}{N(\R)} \rightarrow 0$ as $J \to J_0$.
 
   We have 
    $I_2 = \sum_{(j_1,j_2,j_3) \in [1,J]^3 \setminus \br{j_1=j_2=j_3} } A_{j_1} \bar{A}_{j_2} A_{j_3}$, where
 $A_j = T_{\yn{j}}\Pv_j(t+\tn{j})$.
    Therefore, by H\"older's inequality, we have
    \begin{align*}
      \Xnorm{I_2}{N(\R)} \le \sum_{j_1 \neq j_2} \Xnorm{A_{j_1} \bar{A}_{j_2}}{L^4L^2} \Xnorm{\Pv_{j_3}}{S(\R)}
      + \sum_{j_2 \neq j_3} \Xnorm{A_{j_3} \bar{A}_{j_2}}{L^4L^2} \Xnorm{\Pv_{j_1}}{S(\R)}.
    \end{align*}
  We see from \eqref{eq:6.8b} that 
  $   \sup_J \lim_{n \to \infty} \Xnorm{I_2}{N(\R)} = 0.
 $   Thus, we obtain
     \begin{align}\label{eq:6.22}
      \lims_{n \to \infty} \Xnorm{\bm{\Lambda}_n^J}{N(\R)}  \rightarrow 0.
    \end{align}
    as $J \to J_0$.
    
    Let $\ep_0$ be the constant given by Theorem $\ref{thm.3.14}$ with the choice $M=2C_0$, where $C_0$ is the constant in $\eqref{eq:6.17}$.
    Then $\eqref{eq:6.17}$ implies that
   for all $J$ there exists  $ \tilde{N}_1 =  \tilde{N}_1(J)$ such that 
    if $n \ge \tilde{N}_1$ then $$\Xnorm{\tilde{\wv}_n^{J}}{S(\R)} \le M.$$
    By $\eqref{eq:6.22}$, there exist $J_0$ 
  and $\tilde{N}_2=  \tilde{N}_2(J_0)$ such that 
    if $ n \ge \tilde{N}_2$ then $$\Xnorm{\bm{\Lambda}_n^{J_0}}{N(\R)} \le \tfrac{\ep_0}{2}.$$ 
    Furthermore, by $\eqref{eq:6.18}$, there exists $\tilde{N}_3=  \tilde{N}_3(J_0)$ such that 
    if $n \ge \tilde{N}_3$ then $$\Xnorm{U(t)(\wzero - \tilde{\wv}_n^{J})}{S(\R)} \le \tfrac{\ep_0}{2}.$$ 
    Then,  
    if $n \ge \max(\tilde{N}_1(J_0),\tilde{N}_2(J_0),\tilde{N}_3(J_0))$ then
    $
      \Xnorm{\wv_n - \tilde{\wv}_n^{J_0} }{S(\R)} \le C \ep_0 ,
    $  
    which implies $  \lims_{n \to \infty} \Xnorm{\wv_n }{S(\R)}<\infty$.
    This contradicts $\eqref{eq:6.6d}$. So, we have $J_0=1$. 
  \end{proof}
  By $J_0=1$, one sees that $\eqref{eq:6.8a}$ and $\eqref{eq:6.8c}$ satisfy
  \begin{align}
    &\wzero = U(\tn{1})T_{\yn{1}} \phiv_1 + \rv{1},  \label{eq:6.23a} \\
    &\lim_{n \to \infty} \Xnorm{U(t)\rv{1}}{S(\R)} = 0.  \label{eq:6.23b}
  \end{align}
Arguing as in the previous lemma, we also obtain
    $t_n^1 = 0$.
  Indeed, if $t_n^1 \to \infty$ (resp. if $t_n^1 \to -\infty$) as $n\to\infty$ then
 we obtain $\lim_{n\to\infty}\Xnorm{\wv_n}{S([0,\infty))}=0 $ 
  (resp. $\lim_{n\to\infty}\Xnorm{\wv_n}{S((-\infty,0])}=0 $).
  In particular, we see that $\Pv_1$ is defined by the initial condition $\eqref{eq:6.14a}$. 

Let us prove that $\Pv_1$ is the desired solution $\wv_\infty$.
By a similar argument as in the proof of Lemma \ref{lem.6.8}, one also obtain 
\[
	   \Xnorm{\Pv_1}{S([0,\infty))} = \Xnorm{\Pv_1}{S((-\infty,0])} = \infty,
\]
which is \eqref{eq:6.7d}.
Further, this implies $ M(\Pv_1)E(\Pv_1) \ge  A_c$ by definitions of $L$ and $A_c$.
Plugging this inequality to \eqref{eq:6.12}, we obtain \eqref{eq:6.7a}.
Further, thanks to \eqref{eq:4.7.7}, we obtain
$\rv{1} \to 0$
strongly in $\Hsp$ as $n\to\infty$.
  This reads as
  \begin{align*}
    T_{\yn{1}}^{-1} \wzero \rightarrow \phiv_1  \ \ \ \text{in} \ \ \Hsp
  \end{align*}
  as $n \to \infty$. Hence, \eqref{eq:6.7o} holds.
It follows from Lemma \ref{lem.6.6} that
  $K(\wv_{\infty})(0) =K(\phiv_1) >0$. Thus, by Corollary $\ref{cor.5.11}$, we have
  \begin{align*}
    \sup_{t\in \R} M(\Pv_1(t))H(\Pv_1(t)) < M(\Qv)H(\Qv).
  \end{align*}
  Therefore, $\eqref{eq:6.7b}$ holds. 
  Note that \eqref{eq:6.7c} is a consequence of \eqref{eq:6.7b} and the mass conservation.
  Thus, the proof of Propsition $\ref{prop.6.5}$ is completed.
  \end{proof}

\subsection{Analysis of the critical element}
We investigate the property of $\uc$ in Theorem $\ref{thm.6.4}$. 
Recall that we are supposing that $\eqref{eq:6.1}$. 
Let $\uc=(u_{c,1}, \cdots ,u_{c,N})$ be the solution in Theorem $\ref{thm.6.4}$. 

The first one is the precompactness modulo space translation.
\begin{proposition}[]\label{prop.6.12}
  There exists $y(t)\colon \R \to \R^3$ such that  the set
  $\br{(\Ty{y(t)}u_{c,1}(t),\cdots ,\Ty{y(t)}u_{c,N}(t))  \mid t \in \R}$ is precompact in $\Hsp$. 
\end{proposition}

\begin{proof}[Proof of Proposition \ref{prop.6.12}] 
By \cite[Proposition 3.2]{Duyckaerts-Holmer-Roudenko} and Tychonoff's theorem, this proof is completed. 
\end{proof}
The second one is the zero-moment property.
\begin{theorem}[]\label{thm.6.13}
  $P(\uc)=0$. 
\end{theorem}
\begin{proof}[Proof of Theorem \ref{thm.6.13}] 
  Let $\xi_0 \in \R^3$. Define 
  \begin{align*}
    \wv_c(t) = 
    e^{i x \cdot \xi_0} e^{-it |\xi_0|^2} T_{2 \xi_0t}\uv_{c} (t)
  .
  \end{align*}
  Since 
  $F_j(\wv_c) = e^{i x  \cdot \xi_0} e^{-it |\xi_0|^2} T_{2 \xi_0t} F_j(\uc)$ follows from the gauge condition, 
  one verifies that $\wv_c$  is a solution to (gNLS). 
A computation shows
  \begin{align*}
    H(\wv_c)={}& |\xi_0|^2 M(\uc) + \xi_0 P(\uc) +H(\uc),\\
    E(\wv_c)={}& |\xi_0|^2 M(\uc) + \xi_0 P(\uc) +E(\uc).
  \end{align*}
  We regard $E(\wv_c)$ as a function of $\xi_0$. Then,
  it takes its minimum at $\xi_0 = -\frac{P(\uc)}{2M(\uc)}$. The minimum value is
  $E(\wv_c) = E(\uc) - \tfrac{|P(\uc)|^2}{4M(\uc)} .$
  Suppose that $P(\uc) \neq 0$. 
  Then, choosing $\xi_0$ as above, one has
  \begin{align*}
    M(\wv_c) E(\wv_c) < M(\uc) E(\uc) <M(\Qv)E(\Qv).
  \end{align*}
  Similarly, 
  \begin{align*}
    M(\wv_c) H(\wv_c) < M(\uc) H(\uc) <M(\Qv)H(\Qv).
  \end{align*}
 Then,  $K(\wv_c(t)) >0$ follows from Corollary $\ref{cor.5.11}$.
  Furthermore, $\Xnorm{\wv_c}{S(\R)} =\Xnorm{\uc}{S(\R)} = \infty$ holds. 
  Therefore, we see that
  $L( M(\wv_c) E(\wv_c)) = \infty$. This shows $A_c \le M(\wv_c) E(\wv_c)$.
  Thus, by Proposition $\ref{prop.6.2}$ (3), we have
  \begin{align*}
    A_c \le M(\wv_c) E(\wv_c) < M(\uc) E(\uc) = A_c.
  \end{align*}
  This is a contradiction. 
\end{proof}

\subsection{Nonexistence of the critical element}
Finally, we shall see that the existence of $\uc$ leads us to a contradiction.  
This implies that $\eqref{eq:6.1}$ is false, 
that is, $A_c =M(\Qv)E(\Qv)$ holds. 

\begin{lemma}[]\label{lem.6.16}
For every $\ep > 0$,  there exists $R>0$ such that 
\begin{align*}
  \sum_{j=1}^{N}\int_{|x + y(t)| \ge R} (|\nabla u_{c,j}(t,x)|^2+|u_{c,j}(t,x)|^2+|u_{c,j}(t,x)|^4) dx < \ep
\end{align*}
for all $t \in \R$.
\end{lemma}
This follows from the well-known characterization of the precompactness. 

\begin{proposition}[]\label{prop.6.14}
  $\displaystyle \lim_{t \to \infty} \left|\tfrac{y(t)}{t} \right| = 0$ .
\end{proposition}
\begin{proof}[Proof of Proposition \ref{prop.6.14}] 
  Suppose for contradiction that
  \begin{align*}
    \exists \br{t_n} \to \infty ,\ \exists \ep_0 > 0 \ \text{s.t.}\ \tfrac{|y(t_n)|}{t_n} \ge \ep_0\ \ (t_n \to \infty).
  \end{align*}
 One may assume $y(0)=0$ without loss of generality. For $R>0$, we define $t_0(R) \coloneqq \inf \br{t \ge 0 \mid |y(t)| \ge R}$. Then, one has
  \begin{align*}
    t_0(R)>0,\ |y(t)| < R \text{ for }t<t_0(R),\ |y(t_0(R))| = R. 
  \end{align*}
  We define $R_n$ and $\tilde{t}_n$ by
  $R_n \coloneqq  |y(t_n)|$ and $\tilde{t}_n \coloneqq t_0(R_n)$, respectively. 
  Then, since one has $t_n \ge \tilde{t}_n$ by definition, it holds that 
  $
    \frac{R_n}{\tilde{t}_n} \ge \frac{R_n}{t_n} \ge \ep_0.
  $ 
  In addition, $\tilde{t}_n \to \infty$ $(n\to\infty)$ follows from $R_n \to \infty$ $(n\to\infty)$. 
  Thus, we have constructed
   $\br{\tilde{t}_n} \to \infty$ such that
  \begin{align}\label{eq:6.14.0}
  |y(t)| <R_n& \text{ for }   t \in [0, \tilde{t}_n), &
  |y(\tilde{t}_n)| &{}= R_n, &
  \tfrac{R_n}{\tilde{t}_n} &{}\ge \ep_0.
  \end{align}
  Hereafter, We write $\tilde{t}_n$ for $t_n$. 
  
  By Lemma $\ref{lem.6.16}$, for every $\ep>0$, there exists $R_0(\ep)$ such that 
  \begin{align}\label{eq:6.14.1}
    \sum_{j=1}^{N}
    \int_{|x + y(t)| \ge R_0(\ep)} |\nabla u_{c,j}(t,x)|^2+|u_{c,j}(t,x)|^2 dx < \ep
  \end{align}
  for all $t \in \R$.
  Pick $\theta (x) \in C_0 ^{\infty} (\R)$ such that
  \begin{align*}
    \theta (x) = \begin{cases}
      x & -1 \le x \le 1, \\
      0 & |x| \ge 2^{1/3}, 
    \end{cases}
  \end{align*}
  $|\theta (x)| \le |x|$ on $\R$, and $\Xnorm{\theta '}{L^{\infty}} \le 4$. 
  For $x=(x_1,x_2,x_3) \in \R^3$, let $\phi (x) = (\theta (x_1),\theta (x_2),\theta (x_3))$. 
  Then $|\phi(x)| = x$ for  $|x| \le 1$ and $\Xnorm{\phi}{L^{\infty}} \le 2$.
  For $R>0$, we define $\phi_R(x) = R \phi (x/R)$. Note that
  \[
  	\phi_R(x) = x \quad \text{for } |x| \le R
  \]
  and $\Xnorm{\phi_R}{L^{\infty}} \le 2 R$ and $\Xnorm{\nabla \phi_R}{L^{\infty}} \le 4$.
  We define $z_R\colon \R \to \R^3$ by
  \begin{align*}
    z_R(t) = \int \phi_R\sum_{j=1}^{N} |u_{c,j}|^2  dx . 
  \end{align*} 
  Then, by $\eqref{eq:1.2}$, we have $z_R'(t) = ((z_R'(t))_1, (z_R'(t))_2, (z_R'(t))_3)$ with
  \begin{align*}
    ((z_R'(t))_k = 2 \im \int \partial_k \phi_R 
    ( \sum_{j=1}^{N} \partial_k u_{c,j} \overline{u_{c,j}})dx.
  \end{align*}
  By Theorem $\ref{thm.6.13}$, we have
  \begin{align*}
    \sum_{j=1}^{N} \im \int_{|x_k| \le R} \partial_k u_{c,j} \overline{u_{c,j}} dx 
    = - \sum_{j=1}^{N} \im \int_{|x_k| > R} \partial_k u_{c,j} \overline{u_{c,j}} dx.
  \end{align*}
  Next, we calculate $z_R'(t)$. 
  Multiplying $\Equ{\uv(t,x)} + \bm{F}(\uv(t,x)) =0$ by $\phi_R \overline{\uv}$, taking the real part, and integrating with respect to $x$, 
  we have 
  \begin{align*}
    ((z_R'(t))_k = \sum_{j=1}^{N} \left[
      -2 \im \int_{|x_k| \ge R} \partial_k u_{c,j} \overline{u_{c,j}} dx
    + 2 \im \int_{|x_k| \ge R} 
    \partial_k \phi_R \partial_k u_{c,j} \overline{u_{c,j}} dx
    \right]
  \end{align*}
  for $k=1,2,3$. 
  Note that the nonlinear part are calculated by \eqref{eq:1.2}.
  In addition, using Schwarz's inequality and arithmetic-geometric mean, we have
  \begin{align}\label{eq:6.14.2}
    |z'_R(t)| \le 5 \sum_{j=1}^{N} \int_{|x| \ge R} (|\nabla u_{c,j}|^2 + |u_{c,j}|^2)dx.
  \end{align}
  Let $\tilde{R}_n \coloneqq R_n + R_0(\ep)$.  
  For $t \in [0,t_n]$ and $|x|>\tilde{R}_n$, $|x+y(t)|\ge \tilde{R}_n - R_n =R_0(\ep)$ follows.
  Using $\eqref{eq:6.14.1}$ and $\eqref{eq:6.14.2}$, we have
  \begin{align}\label{eq:6.14.3}
    |z'_{\tilde{R}_n}(t)| \le 5 \ep.
  \end{align}
  It follows from $y(0)=0$ that
  \begin{align*}
    z_{\tilde{R}_n}(0) 
    &= \sum_{j=1}^{N}  \left[ \int_{|x| \le R_0(\ep)} \phi_{\tilde{R}_n}  |u_{j,0}|^2 dx 
    + \int_{|x+y(0)| \ge R_0(\ep)}  \phi_{\tilde{R}_n} |u_{j,0}|^2 dx .
    \right]
  \end{align*}
 Hence, one sees from $\eqref{eq:6.14.1}$ and the properties of $\phi_{\tilde{R}_n}$ that
  \begin{align}\label{eq:6.14.4}
    |z_{\tilde{R}_n}(0) | \le 2R_0(\ep) M(\uc) + 2 \tilde{R}_n \ep.
  \end{align}
  Therefore, we have
  \begin{align*}
    z_{\tilde{R}_n}(t) 
    &= \sum_{j=1}^{N} 
    \int_{|x+y(t)| \ge R_0(\ep)} \phi_{\tilde{R}_n} |u_{j,c}|^2 dx 
    + \sum_{j=1}^{N} \int_{|x+y(t)| \le R_0(\ep)}  \phi_{\tilde{R}_n} |u_{j,c}|^2 dx \\
    & \eqqcolon \rm{I} + \rm {I\hspace{-.01em}I}.
  \end{align*}
  We evaluate $\rm{I}$. By $\Xnorm{\phi_{\tilde{R}_n} }{L^\infty} \le 2 \tilde{R}_n$, we have
 $   |{\rm{I}}| \le 2 \tilde{R}_n \ep$.
  Next, we evaluate $\rm {I\hspace{-.01em}I}$. 
  From $|x| \le |x +y(t)| + |y(t)| \le R_0(\ep) +R_n = \tilde{R}_n$, we obtain $\phi_{\tilde{R}_n}(x) = x$
  for $|x+y(t)| \le R_0(\ep)$. 
  Hence, we have
  \begin{align*}
    \rm {I\hspace{-.01em}I}
    & = \sum_{j=1}^{N} \left[ \int_{|x+y(t)| \le R_0(\ep)} (x+y(t)) |u_{j,c}|^2 dx 
    - y(t) \int_{|x+y(t)| \le R_0(\ep)} |u_{j,c}|^2 dx \right] \\
    & \eqqcolon  \rm{I\hspace{-.01em}I\hspace{-.01em}I} + \rm{I\hspace{-.01em}V}
  \end{align*}
  and
  \begin{align*}
   \rm{I\hspace{-.01em}V}  &=   - 2y(t) M(u_{j,c}) +y(t)\sum_{j=1}^{N} \left[  
     \int_{|x+y(t)| \ge R_0(\ep)} |u_{j,c}|^2 dx \right] \\
    & \eqqcolon   - 2y(t) M(u_{j,c}) + \rm{V}.
  \end{align*}
  One evaluates
$
    |{\rm{I\hspace{-.01em}I\hspace{-.01em}I}}| \le 2R_0(\ep)M(\uc)$ and  
$    |{\rm{V}}| \le |y(t)| \ep.$
  Thus, we have
  \begin{align*}
    |z_{\tilde{R}_n}(t) | & \ge | 2y(t) M(u_{j,c})|
    - (|\rm{I}|+| \rm{I\hspace{-.01em}I\hspace{-.01em}I}|+| \rm{V}|) \\
    & \ge 2|y(t)| M(\uc) -  R_0(\ep)M(\uc) -3  |y(t)| \ep.
  \end{align*}
  Taking $t=t_n$, we see that
  \begin{align}\label{eq:6.14.5}
    |z_{\tilde{R}_n}(t_n)| \ge \tilde{R}_n (2M(\uc) -3 \ep) -2R_0(\ep) M(\uc).
  \end{align}
 Combining $\eqref{eq:6.14.3}$, $\eqref{eq:6.14.4}$, and $\eqref{eq:6.14.5}$, we obtain
  \begin{align*}
    5 \ep t_n  \ge \int_{0}^{t_n} |z'_{\tilde{R}_n}(t)|dt 
    & \ge \left|  \int_{0}^{t_n} z'_{\tilde{R}_n}(t)dt \right| \\
    & = |z_{\tilde{R}_n}(t) - z_{\tilde{R}_n}(0)| \\
    & \ge \tilde{R}_n (2M(\uc) -5 \ep) -4 R_0(\ep) M(\uc).
  \end{align*}
  Recalling that $R_n \le \tilde{R}_n $, we have
  \begin{align*}
    5 \ep  
    \ge \tfrac{{R}_n}{t_n} (2M(\uc) -5 \ep) -  \tfrac{4 R_0(\ep) M(\uc)}{t_n}.
  \end{align*}
  Using $R_n / t_n \ge \ep_0$, we have
  \begin{align*}
    5 \ep  \ge \tfrac{{R}_n}{t_n} M(\uc)  + \ep_0 (M(\uc) -5 \ep) -  \tfrac{4 R_0(\ep) M(\uc)}{t_n}.
  \end{align*}
  If we choose $\ep = \frac{\ep_0}{5(1+ \ep_0)}M(\uc) $ then $ 5 \ep = \ep_0 (M(\uc) -5 \ep) $ and hence
  \begin{align*}
    \left| \tfrac{y(t_n)}{t_n} \right| \le \tfrac{{R}_n}{t_n} \le \tfrac{4R_0(\ep)}{t_n} \to 0\ \ \ (n \to \infty).
  \end{align*}
  This contradicts the assumption. 
\end{proof}
We are in a position to finish the proof. 

\begin{proposition}[]\label{prop.6.15}
  $\uc$ satisfying Theorem $\ref{thm.6.4}$ does not exist.  
\end{proposition}
\begin{proof}[Proof of Proposition \ref{prop.6.15}] 
The proof is done by a standard truncated virial estimate.
To this end, let us introduce two key identities for our nonlinearity.
By $F_j = \partial_{\overline{z_j}}g$, one has
 \begin{equation}\label{E:6.15pf1}
  4 \sum_{j=1}^N \Re (F_j(\uv_c)\overline{u_{c,j}} )
   =  \tfrac{h}{dh} g(h\uv_c)|_{h=1}
   = 4 g(\uv_c).
  \end{equation}
and
\begin{equation} \label{E:6.15pf2}
  4 \sum_{j=1}^N \Re (F_j(\uv_c)\overline{\nabla u_{c,j}} )
   =  \nabla g(\uv_c).
\end{equation}
Note that the first identity is a special case of Euler's homogeneous function theorem.

  Pick $\varphi \in C_0^{\infty}(\R^3)$ such that $\varphi = |x|^2$ for $|x| \le 1$
  and $\varphi =0$ for $|x|\ge 2$.
  For $R>0$, let
  \begin{align*}
    V_R(t) \coloneqq \int R^2 \varphi \left(\frac{x}{R} \right) (\sum_{j=1}^{N} |u_{c,j}|^2) dx.
  \end{align*} 
  Then, in the same way as in the proof of Proposition \ref{prop.6.14}, we have 
  \begin{align*}
    V'_R(t) = 2 \im \int R (\nabla \varphi) \left(\frac{x}{R} \right)  \cdot 
    (\sum_{j=1}^{N} \nabla u_{c,j} \overline{u_{c,j}}) dx.
  \end{align*}
  By H\"older's inequality, we have the upper bound on $|V'_R|$:
  \begin{align}\label{eq:6.5}
    |V'_R(t)| & \le 4R \max_{|x| \le 2R} (\varphi \left(\frac{x}{R} \right) ) 
    \int_{|x| \le 2R} (\sum_{j=1}^{N} |\nabla u_{c,j}| |{u_{c,j}}|) dx \\ \notag
    & \le C R \Xnorm{\nabla \uv(t)}{(L^2)^N} \Xnorm{\uv(t)}{(L^2)^N}.
  \end{align}
  
	Next, we estimate $V''_R$.
By using \eqref{E:6.15pf1} and \eqref{E:6.15pf1}, one obtains 
  \begin{align*}
    V''_R(t) 
      =8 K(\uc) + A_R,
  \end{align*}
where
   \begin{align*}
  A_R :=
       \sum_{j=1}^{N} &
    \left\{
      4 \sum_{l} \int \left(( \partial_{l} ^2 \varphi) \left(\tfrac{x}{R} \right) -2 \right) 
      \left| \partial_{l} u_{c,j} \right|^2 dx  
      \right. \\  & \quad \left.
      + 4 \sum_{l \neq k} \int_{R \le |x| \le 2R} 
      (\partial_k \partial_l \varphi) \left(\tfrac{x}{R} \right)
      \partial_{l} u_{c,j}   \overline{\partial_{k}u_{c,k}} dx 
    \right. \\  & \quad \left.
   - R^{-2} \int (\Delta ^2 \varphi) \left(\tfrac{x}{R} \right) |u_{c,j}|^2 dx 
    \right\}
         - 2\int ((\Delta  \varphi) \left(\tfrac{x}{R} \right) -6) g(\uv_c) dx
  \end{align*}
 We see from the definition of $\varphi \in C_0^\infty$ and $|g(z)| \lesssim |z|^4$ that
  \begin{align} \label{eq:6.7}
    |A_R|  \le C \int_{|x| \ge R} \sum_{j=1}^{N} \left( 
    |\nabla u_{c,j}|^2 +R^{-2} |u_{c,j}|^2 + |u_{c,j}|^4 
    \right) dx.
  \end{align}
  By Corollary $\ref{cor.5.11}$, setting $\delta = 1 - M(\uc)E(\uc)/M(\Qv)E(\Qv)$, we have  
  \begin{align}\label{eq:6.8}
     K(\uc(t)) \ge  \tilde\delta H(\uc(t)) \ge  \tilde\delta E(\uc)>0
  \end{align}
  on $\R$.
  By Lemma $\ref{lem.6.16}$, 
  one sees that there exists
  $R_0 >0$ such that 
  \begin{align}\label{eq:6.9}
    \int_{|x+y(t)| \ge R_0} \left( \sum_{j=1}^{N}
      |\nabla u_{c,j}|^2 +R^{-2} |u_{c,j}|^2  + |u_{c,j}|^4  
    \right) dx \le \tfrac{\tilde\delta}{C}E(\uc),
  \end{align}
  where $C$ is the constant in \eqref{eq:6.7}.
  
  Let $t_0$ and $t_1$ be numbers such that $1 \le  t_0 < t_1 <\infty$ to be chosen later. 
  If we choose $R>0$ so that $R \ge R_0 + \sup_{t \in [t_0,\,t_1]} |y(t)|$ 
  then, $\{ |x|\ge R\} \subset \{ |x+y(t)| \ge R_0\}$ holds
  for any $t \in [t_0,\,t_1]$, and hence
  \begin{align}\label{eq:6.10}
    |V''_R(t)| \ge 7 \delta E(\uc)
  \end{align}
 is valid on $[t_0,\,t_1]$.
  Let $\eta >0$ to be chosen later. By Proposition $\ref{prop.6.14}$, one sees that there exists $t'$ such that $|y(t)| \le \eta t$ holds for any $t \ge t'$.

 Now, we define $t_0=t'$ and 
  $R=R(t_1):=R_0 + \eta t_1$.
 Note that, with the choice, \eqref{eq:6.10} holds for any $t_1>t_0$.
 Then 
  an integration of $\eqref{eq:6.10}$ over $[t_0,\,t_1]$ gives us
  \begin{align}\label{eq:6.11}
    |V'_R(t_1)- V'_R(t_0)| \ge 7 \delta E(\uc)(t_1-t_0).
  \end{align}
 $\eqref{eq:6.5}$ and $\eqref{eq:6.11}$ imply
$ 
    7 \delta E(\uc)(t_1-t_0) \le 2 C \Xnorm{\nabla \Qv}{L^2} \Xnorm{\Qv}{L^2} (R_0 + \eta t_1) 
.$ 
  Let $\eta = \frac{\delta E(\uc)}{C \Xnorm{\nabla \Qv}{L^2} \Xnorm{\Qv}{L^2} }$. Then
$  E(\uc) (5 t_1 - 7t_0) \le C.
 $ Taking $t_1 \to \infty$, we obtain $E(\uc) = 0$. 
  This implies $A_c=M(\uc)E(\uc)=0$, a contradiction.
\end{proof}

\subsection*{Acknowledgement}
S. M. was supported by JSPS KAKENHI Grant Numbers  	21H00991 and 	21H00993.





\end{document}